\documentclass[12pt]{amsart}
\usepackage{amsmath}
\usepackage{amssymb}
\usepackage{amsthm}
\usepackage{enumerate}
\usepackage{url}
\usepackage{comment}
\usepackage[margin=1in]{geometry}
\newcommand{\ad}{\mathrm{ad}}
\newcommand{\adstar}{\ad^*}
\newcommand{\Ad}{\mathrm{Ad}}
\newcommand{\Adstar}{\Ad^*}

\newcommand{\diver}{\mathrm{div}}

\newtheorem{thm}{Theorem}[section]
\newtheorem{prop}[thm]{Proposition}
\newtheorem{lem}[thm]{Lemma}

\newtheorem*{definition}{Definition}
\newtheorem{example}[thm]{Example}
\title{Contactomorphisms with $L^2$ metric on stream functions}
\author{Boramey Chhay}
\address{University of Colorado}
\email{boramey.chhay@colorado.edu}

\date{\today}

\begin{document}

\begin{abstract}
Here we investigate some geometric properties of the contactomorphism group $\mathcal{D}_\theta(M)$ of a compact contact manifold with the $L^2$ metric on the stream functions. Viewing this group as a generalization to the $\mathcal{D}(S^1)$, the diffeomorphism group of the circle, we show that its sectional curvature is always non-negative and that the the Riemannian exponential map is not locally $C^1$. Lastly, we show that the quantomorphism group is a totally geodesic submanifold of $\mathcal{D}_\theta(M)$ and talk about its Riemannian submersion onto the symplectomorphism group of the Boothby-Wang quotient of $M$.
\end{abstract}

\maketitle

\section{Introduction}
Let $M$ be an orientable compact, contact manifold (without boundary) of odd dimension $2n+1$. Recall that a manifold $M$ is a \emph{contact manifold} if there exists a $1$-form $\theta$ which satisfies the non-degeneracy condition that $\theta \wedge d\theta^n\neq 0$ everywhere\cite{B}. We call $\theta$ the \emph{contact form}. If we let $\mathcal{D}(M)$ be the group of diffeomorphisms of $M$, we say that $\eta\in\mathcal{D}(M)$ is a \emph{contactomorphism} if $\eta^*\theta$ is some positive functional multiple of $\theta$. We will denote the contactomorphism group by $\mathcal{D}_\theta(M)$. $\mathcal{D}_\theta(M)$ can be thought of as an infinite dimensional Riemannian manifold using the framework of Arnold\cite{A}. 

The diffeomorphism group of the circle, $\mathcal{D}(S^1)$, has been heavily studied and has interesting applications to fluid mechanics. Depending on the metric, some classical PDE arise as the geodesic equation on $\mathcal{D}(S^1)$ such as the right-invariant Burgers' equation and the Camassa-Holm equation. It was shown that the Riemannian exponential map is not a local $C^1$ map for the $L^2$ metric\cite{CK}. This is not the case when they considered the $H^1$ metric. Later it was shown that $\mathcal{D}(S^1)$ has vanishing geodesic distance for the $H^s$ metric if and only if $s\leq 1/2$ \cite{BBHM,MM}.

The contactomorphism group has been studied before but in many different contexts. Smolentsev \cite{S1,S2} worked with the quantomorphism group $\mathcal{D}_q(M)$, which is the group of diffeomorphisms which exactly preserve the contact form, with the bi-invariant $L^2$ metric on stream functions. In \cite{CP,EP}, $\mathcal{D}_\theta(M)$ was studied with the $L^2$ metric on velocity fields which in turn becomes the full $H^1$ metric on stream functions. In this paper, we consider $\mathcal{D}_\theta(M)$ with the $L^2$ metric on its stream functions. The geometric differences of $\mathcal{D}_\theta(M)$ with these two metrics are apparent just as in the case of $\mathcal{D}(S^1)$. As $\mathcal{D}(S^1)$ coincides with $\mathcal{D}_\theta(S^1)$ trivially, we view $\mathcal{D}_\theta (M)$ as a natural generalization to $\mathcal{D}(S^1)$. In \cite{SH}, Shelukhin considers the $L^\infty$ norm on the contactomorphisms isotopic to the identity and shows how that induces a bi-invariant distance function on the full $\mathcal{D}_\theta(M)$.

We summarize the results of this paper as follows. First we show that $\mathcal{D}_\theta(M)$ has non-negative sectional curvature. Next we prove that the Riemannian exponential map is not a local $C^1$ map. Lastly, we  show that the quantomorphism group $\mathcal{D}_q(M)$ is a totally geodesic submanifold of $\mathcal{D}_\theta(M)$.

I would like to thank Martin Bauer for suggesting this idea and also my advisor Stephen C. Preston for the helpful discussion. I gratefully acknowledge the support of ESI in Vienna and the Simons Foundation Collaboration Grant \#318969.

\section{Geometric Background}
We will be working primarily on the Lie algebra of $\mathcal{D}_\theta(M)$, and we will use the following well-known
fact that the Lie algebra  $T_e\mathcal{D}_{\theta}(M)$ can be identified with the space of smooth functions $f\colon M\to \mathbb{R}$.

\begin{prop}\cite{EP} The Lie algebra $T_e\mathcal{D}_\theta(M)$ consists
of vector fields $u$ such that $\mathcal{L}_u\theta=\lambda \theta$ for some function
$\lambda:M\rightarrow \mathbb{R}$. Any such field is uniquely determined by the function
$f=\theta(u)$, and we write $u=S_\theta f$. Thus we have that
\[T_e\mathcal{D}_\theta(M)=\{S_\theta f:f\in C^\infty(M)\}.\]
\end{prop}


Here we call $S_\theta$ the contact operator. The Lie bracket on $T_e\mathcal{D}_\theta(M)$ is given by
\begin{equation}\label{contactbracket}
[S_\theta f,S_\theta g]=S_\theta \{f,g\},\text{    where   }\{f,g\}=S_\theta f(g)-gE(f);
\end{equation}
here $E$ denotes the Reeb vector field, uniquely specified by the conditions $\theta(E)=1$, $\iota_Ed\theta=0$.
We call $\{\cdot,\cdot\}$ the ``contact Poisson bracket''; it is not a true Poisson bracket
since it does not satisfy Leibniz's rule. 

We also need a Riemannian structure on $(M,\theta)$, and we will require
that the Riemannian metric be \emph{associated} to the contact form. It will also be convenient to assume that $E$ is a Killing field (i.e., its flow consists of isometries).

\begin{definition}
If $(M,\theta)$ is a contact manifold and $E$ is the Reeb field, a Riemannian metric $(\cdot,\cdot)_g$ is \emph{associated}
if it satisfies the following conditions:
\begin{enumerate}
\item $\theta(u)=(u,E)_g$ for all $u\in TM$, and
\item there exists a $(1,1)$-tensor field $\phi$ such that $\phi^2(u)=-u+\theta(u)E$
and $d\theta(u,v)=(u,\phi v)_g$ for all $u$ and $v$.
\end{enumerate}
If in addition $E$ is a Killing field, we say that that $(M,\theta,g)$ is \emph{$K$-contact}.
\end{definition}

Now if we have a $K$-contact manifold $(M,\theta,g)$, we define a right-invariant metric $\langle\cdot,\cdot\rangle$ on $\mathcal{D}_\theta(M)$ by 
\begin{equation}\label{metric}\langle S_\theta f,S_\theta g\rangle=\int_M fgd\mu\end{equation}

\begin{lem} With $X=S_\theta f$ and $Y=S_\theta g$ we have that 
\[\ad^*_XY=S_\theta[S_\theta f(g)+g(n+2)E(f)]\]

\end{lem}
\begin{proof} Let $X=S_\theta f$, $Y=S_\theta g$, and $Z=S_\theta h$ so we have
\begin{multline}
\langle \adstar_XY,Z\rangle= \langle ad^*_{S_\theta f}S_\theta g, S_\theta h\rangle =\langle S_\theta g,\ad_{S_\theta f}S_\theta h\rangle=-\int_M gS_\theta f(h)d\mu+\int_M gE(f)hd\mu=\\\int_MhS_\theta f(g)+hg\diver S_\theta f +hgE(f) d\mu =\int_M (S_\theta f(g)+g(n+1)E(f)+gE(f))hd\mu. \end{multline}
Thus we have that 
\begin{equation}\label{adstar}\ad^*_XY=S_\theta[S_\theta f(g)+g(n+2)E(f)]\end{equation}

\end{proof}

On any Lie group with a right-invariant Riemannian metric, the geodesic equation\cite{A} can be written in terms of the flow equation 
\[\frac{d\eta}{dt}=u\circ \eta\]
and the Euler-Arnold equation
\[\frac{du}{dt}+\adstar_uu=0.\]
 In this case, the Euler-Arnold equation becomes
\[\tfrac{df}{dt}+f(n+3)E(f)=0.\]

\begin{example}
For $M=S^1$ with the coordinate being $\alpha$ and the standard $1-$form being $d\alpha$ we get that the Reeb field is $E=\tfrac{d}{d\alpha}$ and the contact operator is $S_\theta f=fE$. Thus the geodesic equation on the circle becomes 
\[\tfrac{df}{dt}+3ff_\alpha=0.\]
This is the right-invariant Burgers' equation which is studied in \cite{CK}. It is usual Euler-Arnold equation on $\mathcal{D}(S^1)$, the diffeomorphism group of the circle.
\end{example}

In the above example of the circle, the diffeomorphism group, which is studied in \cite{CK, MM}, coincides with the contactomorphism group. As we will see later in this paper, the contactomorphism group shares many properties with the diffeomorphism group of the circle with the $L^2$ right-invariant metric. Thus, we view $D_\theta(M)$ as a generalization of $\mathcal{D}(S^1)$.

\section{The Curvature}
In \cite{KM} it was shown that the contactomorphism group is a regular smooth Lie group. The curvature of a Lie group $G$ with right-invariant metric
in the section determined by a pair of vectors $X,Y$ in the Lie algebra $\mathfrak{g}$ is given by the following formula\cite{AK}.

\begin{equation}\label{arnoldformula}
C(X,Y)=\langle d,d\rangle+2\langle a,b \rangle-3\langle a,a\rangle-4\langle B_X,B_Y\rangle
\end{equation} where
\begin{multline*}
2d=B(X,Y)+B(Y,X), \qquad 2b=B(X,Y)-B(Y,X), \\
2a=\ad_XY,\qquad 2B_X=B(X,X),\qquad 2B_Y=B(Y,Y),
\end{multline*}
where $B$ is the
bilinear operator on $\mathfrak{g}$ given by the relation $\langle B(X,Y),Z\rangle=\langle X,\ad_YZ\rangle$, i.e.,
$B(X,Y) = \adstar_YX$.
Note that in terms of the usual Lie bracket of vector fields, we have $\ad_XY = -[X,Y]$. The sectional curvature is then given by the normalization $K(X,Y)=C(X,Y)/\vert X\wedge Y\vert^2$. But here we only care about the sign so we will work with $C$ only. 

Next we will show that the sectional curvature will always be non-negative.

\begin{thm}
The sectional curvature is nonnegative.
\end{thm}
\begin{proof} 









With $X=S_\theta f$ and $Y=S_\theta g$, we use the above formula and \eqref{adstar} to compute  
\begin{align*}
C(X,Y)
&=\tfrac{1}{4}\int_M[(n+3)(fE(g)+gE(f))]^2\\
&\hspace{2 cm}-2[(\{f,g\})(S_\theta g(f)-S_\theta f(g)+(n+2)(fE(g)-gE(f))]\\
&\hspace{4 cm}-3[\{f,g\}^2]-4[(n+3)^2fE(f)gE(g)] d\mu\\
&=\tfrac{1}{4}\int_M [(n+3)(fE(g)+gE(f))]^2\\
&\hspace{2 cm}-2(n+3)\{f,g\}(fE(g)-gE(f))+\{f,g\}^2\\&\hspace{4 cm}-4[(n+3)^2fE(f)gE(g)] d\mu\\
\end{align*}
because $S_\theta g(f)-S_\theta f(g)+fE(g)-gE(f)=2\{f,g\}$ by antisymmetry of the contact Poisson bracket.
Now since \[(fE(g)+gE(f))^2-4fE(f)gE(g)=(fE(g)-gE(f))^2,\]  we have that the non-normalized sectional curvature is given by
\[C(X,Y)=\tfrac{1}{4}\int_M[\{f,g\}-(n+3)(fE(g)-gE(f))]^2d\mu.\]

\end{proof}

Here we can immediately see how the geometry of $\mathcal{D}_\theta(M)$ changes when we consider the $L^2$ metric on stream functions rather than the $H^1$ metric where it was shown in \cite{CP} that the curvature can take on any sign.

\section{Geodesics}

From \eqref{adstar} we have that the flow equation and Euler-Arnold equation are given by 
\begin{align}&\tfrac{\partial\varphi}{\partial t}=u\circ \varphi\\ &\tfrac{\partial f}{\partial t}+3fE(f)=0\end{align}where $u=S_\theta f$.

Let $(x,z)=(x_1,\ldots,x_n,y_1,\ldots,y_n,z)$ be Darboux coordinates for our contact manifold and thus the contact form is given by 
\[\alpha=dz-\sum y_idx_i\] and the Reeb field is given by 
\[E=\tfrac{\partial}{\partial z}.\]
Now given an initial condition, $f_0$, we are able to solve this first order PDE implicitly in these coordinates to get  \[f(t,x,z-3tf_0(x,z))=f_0(x,z).\]
Note that this solution does not describe trajectories.

In \cite{SH}, it was shown that given the metric \eqref{metric}, the energy functional is in fact degenerate. Thus, just as in the case of the diffeomorphism group of the circle with the $L^2$ metric\cite{MM}, $\mathcal{D}_\theta(M)$ has vanishing geodesic distance.

\section{The Exponential Map}
Let $\varphi(t;v)$  be the geodesic starting at the identity and in the direction of $v$. Recall that the exponential map $\exp_p$ on a Riemannian manifold $M$ at a point $p\in M$ is defined by the geodesic flow at time $1$. Explicitly, it is defined as $\exp_p(v)=\varphi(1,v)$. Next we will show that as with the case of the diffeomorphism group of the circle\cite{CK}; $\mathcal{D}_\theta(M)$ with the $L^2$ metric on stream functions has an exponential map which is not locally $C^1$.

\begin{thm}
The Riemannian exponential map of the $L^2$ right invariant metric on stream functions of $\mathcal{D}_\theta(M)$ is not a $C^1$ map from a neighborhood of zero in $T_e\mathcal{D}_\theta(M)$ to $\mathcal{D}_\theta(M)$.
\end{thm}

\begin{proof}
Let's assume for a contradiction that $\exp$ is a $C^1$ map. 

Consider the curve given by $t\mapsto tu_0$ with $t>0$ and $u_0\in T_e\mathcal{D}_\theta(M)$. For $t$ small enough we have that $\exp(tu_0)=\varphi(1;tu_0)=\varphi(t;u_0)$ we compute
\[\left.\frac{d}{dt}\exp(tu_0)\right\vert_{t=0}=\left.\frac{d}{dt}\varphi(t;u_0)\right\vert_{t=0}=u_0\]
so we have that $D\exp(0)$ is the identity.

Now we would like to show that $\exp$ is not invertible in a neighborhood of $u_0\in T_e\mathcal{D}_\theta(M)$ so we consider the Jacobi fields. Let $\eta(t)$ be a smooth geodesic with $\eta(0)=e$ and $\dot{\eta}(0)=u_0$ so that every Jacobi field satisfies 
\[\tfrac{\partial}{\partial t}\left(\Adstar_\eta\Ad_\eta\tfrac{\partial v}{\partial t}\right) +\adstar_{\frac{\partial v}{\partial t}}u_0=0\]
with $J(t)=dL_\eta v$. This equation is obtained by left translating the Jacobi equation\cite{EMP, MP}. Now since $E$ is a steady state solution to the Euler-Arnold equation, we have that its flow is geodesic in $\mathcal{D}_\theta(M)$, and since $E$ is a Killing field, we have that $\eta(t)$ is an isometry of $M$. Thus $\Adstar_\eta\Ad_\eta$ is the identity for all time. Now with $v=S_\theta g$ and $u_0=S_\theta f_0$ for $g,f_0\in C^\infty (M)$ and setting $f_0=c>0$, we can rewrite the Jacobi equation as 
\[\tfrac{\partial^2g}{\partial t}+c(n+2)E(\tfrac{\partial g}{\partial t})=0.\]

We set $w=\frac{\partial g}{\partial t}$ with initial condition $w_0$ and locally, in Darboux coordinates $(x,z)$ we have that the above equation becomes 
\[\tfrac{\partial w}{\partial t}+c(n+2)\tfrac{\partial w}{\partial z}=0\]
thus solving for $g$ we get 
\[g(x,z,t)=\tfrac{1}{(c(n+2))^2}\int_z^{z-c(n+2)t}w_0(x,s)ds.\]

So letting $c_m=\frac{1}{m}$, we have that $w_0=\sin\left(\frac{2\pi m}{n+2}z\right)$ gets annihilated at the points $S_\theta c_m=c_m E$. Thus we have that the $D\exp(c_m E)$ fails to be invertible at points near zero. That is because $c_m$ is a sequence going to zero so in any topology, $c_m E$ also approaches zero. This violates the Inverse Function Theorem which gives us our desired contradiction.
\end{proof}

\section{The Quantomorphism Group}
In this section we will be considering the group of quantomorphisms. That is, the contactomorphisms which exactly preserve the contact form, not just the structure. This can be written as 
\[\mathcal{D}_q(M)=\{\eta\in\mathcal{D}_\theta(M):\eta^*\theta=\theta\}.\]

A contact form is said to be \emph{regular} if the Reeb field induces a free action of the unit circle on $M$. If a contact form is regular, we are able to define the Boothby-Wang quotient\cite{B} manifold $M/S^1=N$ and the $2$-form $d\theta$ can be then used to define a symplectic structure $\omega$ on $N$ by 

\[\pi^*\omega=d\theta,\]
where $\pi:M\rightarrow N$.
\begin{thm}
If $(M,\theta,g)$ is a $K$-contact manifold with Reeb field $E$. If $\theta$ is a regular contact form, then $D_q(M)$ is a closed and totally geodesic submanifold of $D_\theta(M)$.
\end{thm}

\begin{proof}
In order to show that a submanifold is totally geodesic, it is equivalent to show that the second fundamental form vanishes identically. To do so, it suffices to show that $\langle \nabla_uu,v\rangle=0$ whenever $u$ is tangent and $v$ is orthogonal to the submanifold. For a right-invariant metric on a Lie group, we have that $\nabla_uu=\ad^*_uu$. Thus we would like to show that \[\langle u,\ad_uv\rangle=0\]
whenever $u\in T_e\mathcal{D}_q(M)$ and $v\in T_e\mathcal{D}_\theta(M)$ with $v$ orthogonal to $T_e\mathcal{D}_q(M)$.

So let $u=S_\theta f\in T_e\mathcal{D}_q(M)$ and $v=S_\theta g\in T_e\mathcal{D}_\theta(M)$ orthogonal to $T_e\mathcal{D}_q(M)$.
\begin{equation}
\begin{split}
\langle\nabla_uu,v\rangle&=\langle\adstar_uu,v\rangle= \langle u,\ad_uv\rangle\\&=-\int_M (S_\theta f,S_\theta \{f,g\})_gd\mu=-\int_M f\{f,g\}d\mu\\&=-\int_M fS_\theta f(g)d\mu=\int_M g(E(f)+f\diver S_\theta f)d\mu=0
\end{split}
\end{equation}
\end{proof}

From Smolentsev\cite{S1,S2}, we can see that the quantomorphism group admits a Riemannian submersion onto the symplectomorphism group of the Boothby-Wang quotient. Let $(M,\theta,g)$ be a $K$-contact manifold with regular contact form $\theta$. The vector fields of the quantomorphism group, $T_e\mathcal{D}_q(M)$, are those $V\in T_e\mathcal{D}_\theta(M)$ such that $\mathcal{L}_V\theta=0$. Now let $N=M/S^1$ with $\omega$ the induced symplectic structure by $\pi:M\rightarrow N$ and let $D_\omega(N)$ be the group of symplectomorphisms of $N$

\[\mathcal{D}_\omega(N)=\{\eta\in \mathcal{D}(N):\eta^*\omega=\omega\}.\]

Now $T_e\mathcal{D}_\omega(N)$ consists of the vector fields $V$ such that $\mathcal{L}_V\omega=0$. We call a vector field $V$ Hamiltonian if we can associate a function $H$ such that $\omega(\cdot,V)=dH(\cdot)$. In order for this definition to be unambiguous, we require that the Hamiltonians have mean zero.

For $V\in T_e\mathcal{D}_q(M)$, we have that $[V,E]=0$ and thus $T_e\mathcal{D}_q(M)\rightarrow T_e\mathcal{D}_\omega(N)$ is a projection. We can see that elements of $T_e\mathcal{D}_q(M)$ are of the form $V=hE+X$. These vector fields project onto $T_e\mathcal{D}_\omega(N)$ by $d\pi\circ X=Y\circ \pi$ with $Y\in T_e\mathcal{D}_\omega(N)$. Here we have that $\mathcal{L}_V\theta$ implies that $E(h)=0$ so that $h$ is constant in the Reeb direction. Now combined with the fact that we require our stream functions and Hamiltonians to have mean zero, we can see that the map
\[d\pi:\ker(d\pi)^\perp\rightarrow TD_\omega(N)\]
is an isometry by scaling the one of the volume forms by a constant. Thus the projection of $\mathcal{D}_q(M)$ onto $\mathcal{D}_\omega(N)$ is a Riemannian submersion.

\end{document}